\documentclass[12pt]{article}
\usepackage[a4paper, total={6in, 9in}]{geometry}
\usepackage{cite}
\usepackage{amsthm}
\usepackage{amsmath,amssymb,amsfonts}

\usepackage{algorithmic}
\usepackage{graphicx}
\usepackage{algorithm,algorithmic}
\usepackage{hyperref}
\hypersetup{hidelinks=true}
\usepackage{textcomp}

\usepackage{graphicx} 
\usepackage{mathtools}
\usepackage{amsfonts}

\usepackage{comment}
\usepackage{caption}
\usepackage{subcaption}

\usepackage{mathtools}
\DeclarePairedDelimiter\bra{\langle}{\rvert}
\DeclarePairedDelimiter\ket{\lvert}{\rangle}
\DeclarePairedDelimiterX\braket[2]{\langle}{\rangle}{#1\,\delimsize\vert\,\mathopen{}#2}

\newtheorem{theorem}{Theorem}[section]
\newtheorem{remark}[theorem]{Remark}
\newtheorem{definition}[theorem]{Definition}
\newtheorem{lemma}[theorem]{Lemma}
\newtheorem{prop}[theorem]{Proposition}

\newcommand{\uu}{\textbf{z}}
\begin{document}
\title{Enhancing the controllability of quantum systems via a static field
}
\author{Ruikang Liang\thanks{R.~Liang, M.~Sigalotti and Ugo~Boscain are with Sorbonne Universit\'e, Universit\'e 
Paris Cit\'e, CNRS, INRIA, Laboratoire Jacques-Louis Lions, LJLL, F-75005 Paris, France {\tt 
   ruikang.liang@inria.fr, 
  ugo.boscain@inria.fr,
mario.sigalotti@inria.fr}.}, Eugenio Pozzoli\thanks{E.~Pozzoli is with Univ Rennes, CNRS, IRMAR - UMR 6625, F-35000 Rennes, France {\tt eugenio.pozzoli@univ-rennes.fr}}, Monika Leibscher\thanks{M.~Leibscher and C.P.~Koch are with Freie Universit\"at Berlin, Dahlem Center for Complex Quantum Systems and Fachbereich Physik, Arnimallee 14, 14195 Berlin, Germany {\tt  monika.leibscher@fu-berlin.de,
christiane.koch@fu-berlin.de}}, Mario Sigalotti\footnotemark[1],\\
Christiane P. Koch\footnotemark[3], and Ugo Boscain\footnotemark[1]
}

\maketitle

\begin{abstract}
We provide a sufficient condition for the controllability of a bilinear closed quantum system 
{ 
steered by a static field and a time-varying field,
based on the notion 
of 
}
weakly conically connected spectrum. More precisely, we show that if a controlled Hamiltonian with two inputs has a weakly conically connected spectrum, then, freezing one of the two inputs at almost every constant value, the obtained single-input system is controllable. The result is illustrated with two examples,  { enantio-selective} excitation in a chiral molecule and the driven Jaynes--Cummings Hamiltonian.
\end{abstract}

\section{Introduction}
Consider a bilinear controlled closed quantum system
\begin{equation}\label{proto-system}
  { 
  i\dot{\psi}=H(u)\psi=\Big(H_{0}+\sum_{\ell=1}^{m}u_{\ell}H_{\ell}\Big)\psi,}
\end{equation} 
where ${ \psi=}\psi(t)$ belongs to the unit sphere of a finite- or infinite-dimensional { complex} Hilbert space, $H_0$ is the internal Hamiltonian of the system, $H_1,\ldots,H_m$ are operators describing the coupling between the system and the controls ${ u_1
,\ldots u_m
}$.

The controls can account for many different physical interactions (e.g., electric or magnetic fields). 
%
Depending on 
the physical constraints on
the control $u_\ell$ and on the corresponding coupling $H_\ell$, 
it might not be
possible 
to make ${ u_\ell
}$ oscillate at the desired frequency (e.g., the frequency $E_{k}-E_{j}$ if { $H_\ell$} is coupling
the energy levels 
$E_{k}$ and $E_{j}$) and actually $u_\ell$ should be considered almost constant.
Such a situation results in having 
less controls on the system, since some of them should be treated as static. However, the possibility of choosing  the value of such static fields can 
be enough to guarantee
controllability.

This occurs for instance in molecular dynamics, where a constant electric field is applied to lift spectral degeneracies of the rotational drift Hamiltonian. This procedure is usually referred to as the Stark effect. Another external field can then be switched on and off, with a time-dependent frequency, to displace the state from an initial to a target condition: since the system is experiencing a static external field which is supposed to lift the degeneracies, this is now easier since it is more likely to possess a non-degenerate chain of transitions between those states. For such control strategies in the presence of a static field and a physical application we refer, e.g., to the recent work \cite{Leibscher2024},
dealing with the task of creating chiral vibrational wave packets in ensembles of achiral molecules. 
 The dependence of the controllability properties of the system on the strength of the static field is studied in \cite{Leibscher2024} by  using graphical methods accounting for uncoupled 
transitions among energy levels obtained by resonant controls. Such graphical methods extend the enantio-selective controllability analysis for chiral quantum rotors proposed in \cite{Leibscher2022,Pozzoli2022}.



In this paper we propose a novel technique to study the controllability of systems subject both to time-varying and static fields, 
which leverages on the presence of (weakly) conical intersections in 
the spectrum of the Hamiltonian.

Consider a system driven by a real Hamiltonian and by two controls. In \cite{liang:hal-04174206} 
we proved that if the spectrum is weakly conically connected, then the system is exactly controllable in the finite-dimensional case and approximately controllable in the infinite-dimensional one. The notion of weakly conically connected, introduced in \cite{liang:hal-04174206}, generalizes the more 
classical notion of conically connected spectrum, formalized in \cite{boscain2015approximate}. In both cases, the main assumption behind this notion is that the spectrum is discrete and that every two nearby energy levels of the system are connected by an eigenvalue intersection. For conically connected spectra these intersections must be conical, while for weakly conically 
connected spectra it is enough that a conical direction through the intersection exists. Moreover, in the case of weakly conically 
connected spectra we replace the assumption that eigenvalue { intersections} cannot pile up at a common value of the control by another, weaker, assumption, stating that the eigenvalue intersections have rationally unrelated germs in a suitable sense (see Definition~\ref{def:unrelated_germs}). 
Both for conically connected and weakly conically connected spectra, the controllability result is reflected in an adiabatic control strategy to steer the system between two eigenstates 
going through eigenvalue intersections along conical directions.


The main result of the present study is that, under the assumption that the spectrum is weakly conically connected,  
freezing one of the two controls, then,  for almost every choice of its constant 
value,   the system is controllable (exactly for finite-dimensional systems and approximately for infinite-dimensional ones) via the 
other
control.

The price to pay is that, since only one control can be modulated as a function of time, the system cannot be controlled via adiabatic techniques and hence the explicit expression of the 
non-constant
control realizing the transition is not easily computable (although in principle possible via Lie--Galerkin techniques, see, for instance, \cite{Caponigro2012,chambrionautomatica,chambrion-pozzoli}).

The structure of the paper is the following. In Section~\ref{s-ass} we introduce the main assumptions on the class of systems that we are considering. 
In Section~\ref{s-contr} we recall the 
controllability 
results from \cite{liang:hal-04174206}. Section~\ref{s-main} contains the statement and the proof of our main result (Theorem \ref{theorem:main}). 
In Sections~\ref{s-enantio} and \ref{s-jaynes} we apply our theory on a 3-level enantio-selectivity model and on the driven Jaynes--Cummings system.

\section{Assumptions}
\label{s-ass}
We consider the controllability of 
{ a general}
bilinear Schr\"odinger equation
\begin{equation}\label{system}
    i\dot{\psi}(t)=H(u(t))\psi(t)=\Big(H_{0}+\sum_{\ell=1}^{m}u_{\ell}(t)H_{\ell}\Big)\psi(t),
\end{equation} 
where $\psi(t)$ belongs to a complex Hilbert space $\mathcal{H}$, $H(u(t))$ is a self-adjoint operator on $\mathcal{H}$, and the control $u(\cdot)=\big(u_{1}(\cdot),\dotsc,u_{m}(\cdot)\big)$ takes values in an open and connected subset $U$ of $\mathbb{R}^{m}$. We say that System~\eqref{system} satisfies \textbf{(H)} if
\begin{itemize}
    \item $H_{0},\dotsc,H_{m}$ $(m\geq 2)$ are self-adjoint operators on $\mathcal{H}$ with $\dim(\mathcal{H})<\infty$,
\end{itemize}
    and that System~\eqref{system} satisfies $\textbf{(}\textbf{H}^{\infty}\textbf{)}$ if the following assumptions hold true:
\begin{itemize}
    \item $\mathcal{H}$ is a separable Hilbert space with $\dim(\mathcal{H})=\infty$,
    \item $H_{0}$ is a { lower-bounded} self-adjoint operator on $\mathcal{H}$ with domain $D(H_{0})$ and $H_{1},\dotsc,H_{m}$ $(m\geq 2)$ are symmetric operators on $D(H_{0})$,
    \item $H_{0}$ is an operator with compact resolvent, i.e., $(H_{0}+i)^{-1}$ is a compact operator, and $H_{1},\dotsc,H_{m}$ are 
    Kato-small with respect to $H_0$, { i.e., for each $\ell\in\{1,\dots,m\}$, $D(H_0)\subset D(H_{\ell})$ and 
    for every $a>0$ there exists $b\ge 0$ 
such that $\|H_{\ell}v\|\leq a\|H_{0}v\|+b\|v\|$ for all $v\in D(H_{0})$
} (see \cite{kato1966perturbation}).
\end{itemize}
Under assumption \textbf{(H)} or $\textbf{(}\textbf{H}^{\infty}\textbf{)}$, the spectrum of $H(u)$ is discrete for every $u\in U$. 
Let $\mathcal{I}$ denote $\{1,\dotsc, \dim(\mathcal H)\}$ (respectively, $\mathbb{N}^{*}$) 
when \textbf{(H)} (respectively, $\textbf{(}\textbf{H}^{\infty}\textbf{)}$) is satisfied.
For every $u\in U$, we denote by $\{\lambda_{j}(u)\}_{j\in\mathcal{I}}$ the increasing sequence of eigenvalues of $H(u)$, counted according to their multiplicities, and by 
$\{\phi_{j}(u)\}_{j\in\mathcal{I}}$ a corresponding sequence of  eigenvectors  forming an othonormal basis of $\mathcal{H}$.

\section{Controllability of systems with weakly conically connected spectrum: previous results}
\label{s-contr}

Let us recall 
a classical definition of controllability for systems evolving 
in a finite-dimensional Hilbert space.

\begin{definition}[Controllability { on} the unitary group]\label{def:exact_controllability} When $n=\text{dim}(\mathcal{H})<\infty$, the \emph{lift of system~\eqref{system}} on $\mathcal{G}=U(n)$ (or on $\mathcal{G}=SU(n)$ in the case where $iH_{0},\dotsc,iH_{m}\in su(n)$) is defined as
\begin{equation}
\label{lifted_system}
    i\dot{q}(t)=H(u(t))q(t),
\end{equation}
where $q(\cdot)$ takes values in 
$\mathcal{G}$. 
{ System~\eqref{system} is \emph{exactly controllable { on} the unitary group} if its lift~\eqref{lifted_system} has the property that, for every choice of initial state $q_{0}$ and final state $q_{1}$ in $\mathcal{G}$, there exist $T>0$ and a piecewise constant control $u:[0,T]\rightarrow U$ for which the solution of system~\eqref{lifted_system} starting from $q(0)=q_{0}$ satisfies $q(T)=q_1$.}
\end{definition}
{ 
    \begin{remark}
    In Definition~\ref{def:exact_controllability}, we restrict ourselves to piecewise constant controls, as in~\cite{JURDJEVIC1972313} and~\cite{ELASSOUDI1996117}. All the following results on finite-dimensional systems will be stated with respect to this notion of controllability. However, the class of admissible controls can be extended to \(L^{\infty}\) functions, in which case all results clearly remain valid.
\end{remark}}


For systems evolving 
in 
an infinite-dimensional Hilbert space the natural counterpart of the previous definition is the following. 
\begin{definition}[Approximate controllability { on} the unitary group]
\label{def:controllability_infinite_unitary}
   Let $\textbf{(}\textbf{H}^{\infty}\textbf{)}$ be satisfied. 
   For every piecewise constant control $u:[0,T]\rightarrow U$,
    in the form of $u=\sum_{j=1}^{l}u_{j}1_{(t_{j-1},t_{j})}$ with $t_{0}=0$, let the \emph{propagator} of system~\eqref{system} associated with $u$ be given by
    \begin{multline*}
    \Upsilon_{t}^{u}=e^{i(t-t_{l})H(u_{l})}\circ e^{i(t_{l}-t_{l-1})H(u_{l-1})}\circ \dots\circ e^{it_{1}H(u_{1})}\qquad \mbox{for }t_{l}<t\leq t_{l+1}. 
    \end{multline*}
      We say that system~\eqref{system} is \emph{approximately controllable { on} the unitary group} if for every 
      $k\in \mathbb{N}$, $\Upsilon:\mathcal{H}\to\mathcal{H}$ unitary,  
      $\psi_{0},\dots, \psi_{k}$ in the unit sphere of $\mathcal{H}$, and every $\epsilon>0$, there exist a time $T\geq 0$ and a piecewise constant control function $u:[0,T]\rightarrow U$ such that $\|\Upsilon_{T}^{u}(\psi_{j})-\Upsilon(\psi_{j})\|<\epsilon$ for every $j=1,\dots,k$.
\end{definition}

\begin{remark}
    Notice that different authors use different names for the type of controllability defined in Definition  \ref{def:controllability_infinite_unitary}. For instance in \cite{Caponigro2012} it is called 
{\em simultaneous approximate controllability} and in \cite{KZSH14} it is called {\em strong controllability}.
\end{remark}

As explained in the introduction, this paper deals with the problem of understanding 
when the controllability (or approximate controllability) of the system corresponding to the Hamiltonian $H(\cdot,\cdot)$ 
implies that, fixing the control $u_{1}=\bar{u}_{1}$, the system with the Hamiltonian $H(\bar{u}_{1},\cdot)$ is still controllable (or approximately controllable) for almost every $\bar{u}_{1}$. 
Without additional hypotheses, this is not true in general. Consider, for instance, the system corresponding to the  Hamiltonian
\begin{equation*}
    H(u_1,u_2){ =H_0+u_1H_1+u_2H_2}=\begin{pmatrix}
        1 & u_2 & u_1 & 0  \\
        u_2 & 2 & 0 & 2u_1 \\
        u_1 & 0 & 3 & u_2 \\
        0 & 2u_1 & u_2 & 6 
    \end{pmatrix},
\end{equation*}
which is controllable { since $\text{Lie}(iH_0,iH_1,iH_2)$, which is the smallest Lie subalgebra of $u(4)$ containing $\{iH_0,iH_1,iH_2\}$, coincides with $u(4)$ (see \cite{JURDJEVIC1972313})}. However, for all $\bar{u}_1\in\mathbb{R}$, the system with  Hamiltonian $H(\bar{u}_1,\cdot)$ is not controllable. { 
Indeed, denoting by 
 $\{\textbf{e}_{1},\textbf{e}_2,\textbf{e}_3,\textbf{e}_4\}$ the canonical basis of $\mathbb{C}^{4}$,
  the two-dimensional subspace spanned by $\textbf{e}_1$ and $\textbf{e}_3$ (respectively, by $\textbf{e}_2$ and $\textbf{e}_4$) is invariant under $H(\bar{u}_1,\cdot)$. Notice that the spectrum of $H(u_1,u_2)$, as a function of $(u_1,u_2)\in\mathbb{R}^2$, is never connected by conical or weakly conical intersections (see Definitions~\ref{def:conical} and \ref{def:weakly_conical_intersection}).} In the following, we will prove that, if the spectrum of $H(\cdot,\cdot)$ is weakly conically connected and has rationally unrelated germs, we can deduce controllability results not only for the two-input system but also for the single-input system with the Hamiltonian $H(\bar{u}_{1},\cdot)$ for almost every $\bar{u}_{1}$. 

\begin{definition}[Conical intersection]\label{def:conical}
Let $\mathcal I^-=\mathcal I\setminus\{\dim \mathcal{H}\}$ (that is, $\mathcal I^-=\{1,\dots,\dim \mathcal{H}-1\}$ if $\dim \mathcal{H}<\infty$ and 
$\mathcal I^-=\mathcal I$ otherwise).
Given $j\in \mathcal I^-$, an eigenvalue intersection  $\bar{u}\in U$ between two eigenvalues $\lambda_{j}(\cdot)$ and $\lambda_{j+1}(\cdot)$
     is said to be \emph{conical} if it is of multiplicity two and there exist $C,\delta>0$ such that
    \begin{equation*}
        \frac{1}{C}|t|\leq|\lambda_{j+1}(\bar{u}+t\eta)-\lambda_{j}(\bar{u}+t\eta)|\leq C|t|
    \end{equation*}
    for every unit direction $\eta\in\mathbb{R}^{m}$ and $t\in (-\delta,\delta)$.
\end{definition}
See Figure \ref{fig:conical} for an example of conical intersection.
\begin{definition}[Conical direction]\label{def:conical_direction}
Given $j\in \mathcal I^-$ and an intersection $\bar{u}$ between $\lambda_{j}(\cdot)$ and $\lambda_{j+1}(\cdot)$, a unit vector $\eta\in\mathbb{R}^{m}$ is said to be a \emph{conical direction at $\bar u$} if there exist $C,\delta>0$ such that
    \begin{equation*}
        \frac{1}{C}|t|\leq \lambda_{j+1}(\bar{u}+t\eta)-\lambda_{j}(\bar{u}+t\eta)\leq C|t|
    \end{equation*}
    for $t\in (-\delta,\delta)$.
\end{definition}

Let us now recall the notion of weakly conical intersection proposed in \cite{liang:hal-04174206}, which generalizes the notion of semi-conical intersection considered in \cite{semiconical}.

\begin{definition}[Weakly conical intersection]\label{def:weakly_conical_intersection}
A non-conical eigenvalue intersection $\bar{u}\in U$ between two eigenvalues $\lambda_{j}(\cdot)$ and $\lambda_{j+1}(\cdot)$, $j\in \mathcal I^-$,  is said to be \emph{weakly conical} if 
it is of multiplicity two,
$\bar{u}$ is an isolated point in \begin{equation*}
    \left\{u\in U\mid\lambda_{j}(u)=\lambda_{j+1}(u)\right\},
\end{equation*}
and there exists at least one conical direction at $\bar{u}$.
\end{definition}
See Figure \ref{fig:weakly_conical} for an example of weakly conical intersection. 
\begin{figure}[!t]
\centering
     \begin{subfigure}[b]
     {0.45\columnwidth}
         \centering
         \includegraphics[width=\columnwidth]{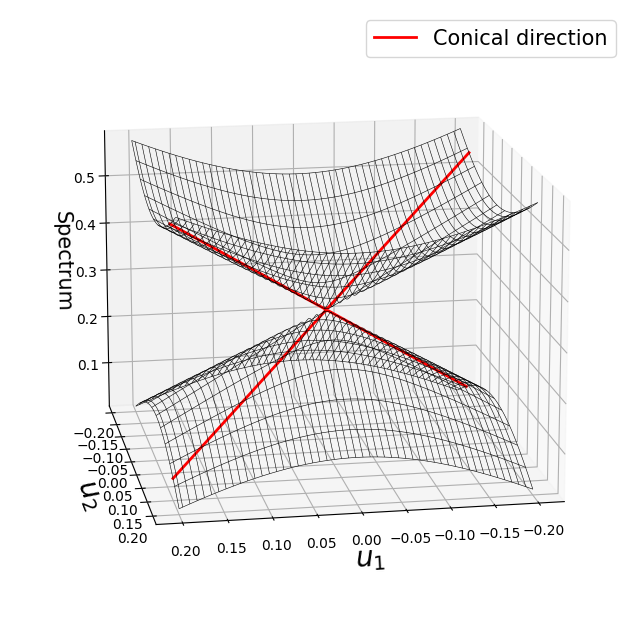}
         \caption{A conical intersection}
         \label{fig:conical}
     \end{subfigure}
     \hfill
     \begin{subfigure}[b]{0.45\columnwidth}
    
         \centering
         \includegraphics[width=\columnwidth]{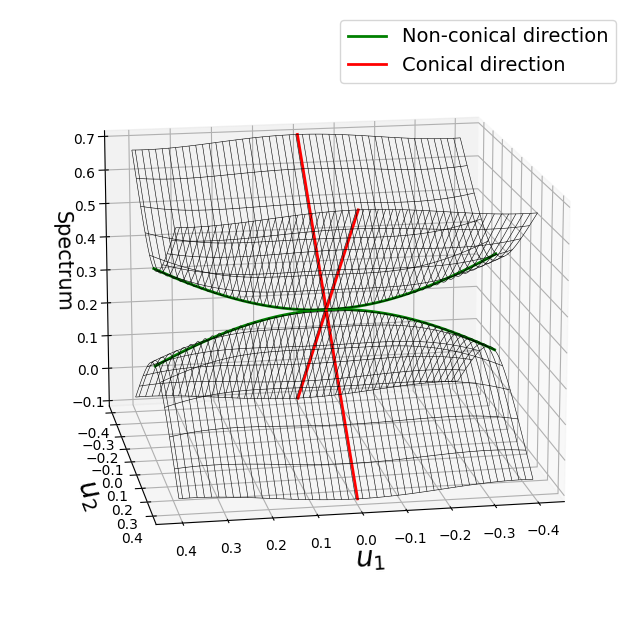}
         \caption{A weakly conical intersection}
         \label{fig:weakly_conical}
     \end{subfigure}
   \label{Fig:example_1}
   \caption{Intersections between eigenvalues of $H(u)$ seen as functions of $u=(u_1,u_2)$, with $m=2$. For the conical intersection, the red lines show the separation of the eigenvalues along a conical direction. For the weakly conical intersection, the red lines show the separation
   of the eigenvalues
   along a conical direction and the green lines show the separation along a non-conical direction. }
\end{figure}

\begin{definition}[Weakly conical connected spectrum]
 Let \textbf{(H)} or $\textbf{(}\textbf{H}^{\infty}\textbf{)}$ be satisfied. We say that 
the spectrum of $H(\cdot)$ is \emph{weakly conically connected} if each of its eigenvalue intersections is either conical or weakly conical, and, moreover, for every $j\in \mathcal I^-$, there exists $\bar{u}_{j}\in U$ such that $\lambda_{j}(\bar{u}_{j})= \lambda_{j+1}(\bar{u}_{j})$.   
\end{definition}

\begin{definition}[Rationally unrelated germs]\label{def:unrelated_germs}
    Let \textbf{(H)} or $\textbf{(}\textbf{H}^{\infty}\textbf{)}$ be satisfied and the spectrum of $H(\cdot)$ be weakly conically connected. 
    For a given 
intersection point $\bar{u}\in U$ of the spectrum of $H(\cdot)$, define 
\begin{equation}\label{eq:I_u}
    \mathcal{I}(\bar{u}):=\{j\in \mathcal{I}^- \mid \lambda_{j}(\bar{u})=\lambda_{j+1}(\bar{u})\}.
\end{equation}
We say that \emph{intersections have rationally unrelated germs at} $\bar{u}$ if for every finite subset $\mathcal{J}$ of $\mathcal{I}(\bar{u})$, and for every neighborhood $V$ of $\bar{u}$, the family of functions 
\begin{equation*}
    \Big(V\ni u\mapsto \lambda_{j+1}(u)-\lambda_{j}(u)\Big)_{j\in\mathcal{J}}
\end{equation*}
is rationally independent, meaning that if 
$\{\alpha_{j}\}_{j\in\mathcal J}\in\mathbb{Q}^{\mathcal{J}}$ is such that 
\begin{equation*}
    \sum_{j\in\mathcal{J}}\alpha_{j}(\lambda_{j+1}(u)-\lambda_{j}(u))=0,\quad \forall u\in V,
\end{equation*}
then $\alpha_{j}=0$ for all $j\in\mathcal{J}$.
\end{definition}
Let us stress that in \cite[Section 4]{liang:hal-04174206} we proposed constructive tests for checking that an eigenvalue intersection has rationally unrelated germs. 
The following spectral conditions for controllability have been proven in \cite{liang:hal-04174206}.
\begin{theorem}[Theorem 1.1 in \cite{liang:hal-04174206}]\label{exact_controllability}
 Assume that System~\eqref{system} satisfies \textbf{(H)}, that the spectrum of $H(\cdot)$ is weakly conically connected, and that
 its eigenvalue intersections have rationally unrelated germs at each intersection point. Then $\text{Lie}(iH_{0},\dotsc,iH_{m})=su(n)$ if $iH_{0},\dotsc,iH_{m}\in su(n)$ and $\text{Lie}(iH_{0},\dotsc,iH_{m})=u(n)$ otherwise, meaning that  System~\eqref{system} is exactly controllable { on} the unitary group.
\end{theorem}
\begin{theorem}[Theorem 3.8 and Remark 3.9 in \cite{liang:hal-04174206}]\label{theorem:approximate_controllability}
    Assume that System~\eqref{system} satisfies $\textbf{(}\textbf{H}^{\infty}\textbf{)}$, that the spectrum of $H(\cdot)$ is weakly conically connected, and that its eigenvalue intersections have rationally unrelated germs at each intersection point. Then System~\eqref{system} is approximately controllable { on} the unitary group.
\end{theorem}

\section{Main result}
\label{s-main}
From now on, let us fix $m=2$. For every possible static value $u_1=\bar{u}_{1}$ in { 
\begin{equation}\label{equation:interval_u1}
I^{1}=\pi_{1}(U),
\end{equation}
where $\pi_{1}$ is the standard projection on the first component,} let us consider the single-input system
\begin{equation}
\label{one_input_system}
    i\dot{\psi}(t)=H(\bar{u}_{1},u_{2}(t))\psi(t)=\Big(H_{0}+\bar{u}_{1}H_{1}+u_{2}(t)H_{2}\Big)\psi(t),
\end{equation}
where $u_2(\cdot)$ takes value in the set $\{u_2\mid (\bar{u}_1,u_2)\in U\}$.  Notice that for each possible static value $\bar{u}_{1}$, $u_{2}(\cdot)$ takes value in an open nonempty set, since $U$ is open in $\mathbb{R}^{2}$. 

Under the 
assumptions of Theorems~\ref{exact_controllability} and \ref{theorem:approximate_controllability}, we can further deduce the controllability of the single-input system with Hamiltonian $H(\bar{u}_{1},\cdot)$.
\begin{theorem}\label{theorem:main}
    Assume that $m=2$, that System~\eqref{system} satisfies \textbf{(H)} (respectively, $\textbf{(}\textbf{H}^{\infty}\textbf{)}$), that the spectrum of $H(\cdot)$ is weakly conically connected, and that its eigenvalue intersections have rationally unrelated germs at each intersection point. Then, for almost every $\bar{u}_{1}\in I^{1}$ defined in \eqref{equation:interval_u1}, the single-input system~\eqref{one_input_system} 
    is controllable (respectively, approximately controllable) { on} the unitary group.
\end{theorem}

{ 
Notice that throughout this paper (including the statement above), we work solely with the \emph{Lebesgue measure}.}

{ Let $\mathcal{I}$ and $\mathcal{I}^-$ be defined as in Section~\ref{s-ass}}. 
For $u\in U$  and $j\in \mathcal{I}$ such that  $\lambda_{j}(u)$ is simple,
denote by 
$P_{j}(u)$ the associated eigenprojection
on the space spanned by the eigenvector $\phi_{j}(u)$. { 
The proof of Theorem~\ref{theorem:main} is structured as follows:
\begin{itemize}
\item  We first prove the intermediate result Lemma~\ref{lemma:one_input_coupling}, which shows that for almost every $\bar{u} \in U$, all eigenvalues of $H(\bar{u})$ are simple, and there is a non-zero coupling via $H_2$ between each pair of eigenvectors $(\phi_j(\bar{u}), \phi_{j+1}(\bar{u}))$ for $j \in \mathcal{I}^{-}$.
\item We then apply  Lemma~\ref{lemma:non_resonance_finite} from \cite{liang:hal-04174206} and deduce that, for almost every $\bar{u}\in U$, the eigenvalues of $H(\bar{u})$ are non-resonant (i.e. $\lambda_{j}(\bar{u})-\lambda_{k}(\bar{u})\neq\lambda_{p}(\bar{u})-\lambda_{q}(\bar{u})$ if $(j,k)\neq(r,s)$ and $j\neq k$, $r\neq s$).
 \item We conclude that for almost every $\bar{u}_1 \in I^1$, there exists some $\bar{u}_2 \in \{u_2 \mid (\bar{u}_1, u_2) \in U\}$ such that the eigenvalues of $H(\bar{u}_1, \bar{u}_2)$ are non-resonant, and there is a non-zero coupling via $H_2$ between each pair of eigenvectors $(\phi_j(\bar{u}_1, \bar{u}_2), \phi_{j+1}(\bar{u}_1, \bar{u}_2))$ for $j \in \mathcal{I}^{-}$. The connectedness of all eigenvectors via $H_2$ and the non-resonance of the spectrum allow us to apply the classical spectral condition for controllability (see \cite[Proposition 11]{boscain2015approximate} and \cite[Proposition 15]{boscain2015approximate}). We conclude that the system characterized by $H(\bar{u}_1, u_2(t))$ is controllable (respectively, approximately controllable) on the unitary group, completing the proof of Theorem~\ref{theorem:main}.
\end{itemize}}
\begin{lemma}\label{lemma:one_input_coupling}
     Let $m=2$, assume that  System~\eqref{system} satisfies $\textbf{(}\textbf{H}\textbf{)}$ or $\textbf{(H}^{\infty})$, and that the spectrum of $H(\cdot)$ is weakly conically connected. Then there exists 
    a subset { $\hat{U}$} of  $U$ with zero-measure complement 
    such that for each { $\bar{u}\in\hat{U}$} all eigenvalues of $H(\bar{u})$ are simple and 
    \begin{equation*}
        \langle\phi_{j}(\bar{u}),H_{2}\phi_{j+1}(\bar{u})\rangle\neq 0,\quad \forall j\in\mathcal{I}^-.
    \end{equation*}
\end{lemma}
\begin{proof}
    For each $j\in\mathcal{I}$, let us define 
    \begin{equation*}
        U_{j}=\left\{u\in U\mid \lambda_{j}(u)\text{ is simple}\right\}.
    \end{equation*}
    Since the spectrum is weakly conically connected, $U_{j}$ is an open and dense subset of $U$ with discrete complement. 
    
    For every $j\in \mathcal{I}^-$,  let us define 
    \begin{equation*}
        Z_{j}=\left\{u\in U_{j}\cap U_{j+1}\mid P_{j}(u)H_{2}P_{j+1}(u)=0 \right\}.
    \end{equation*}
    Because of the analyticity of {$u\mapsto P_{j}(u)H_{2}P_{j+1}(u)$} (cf.~\cite{kato1966perturbation}) on the open and connected set $U_{j}\cap U_{j+1}$, we have either that $Z_{j}=U_{j}\cap U_{j+1}$ or that $Z_{j}$ has empty interior and  zero measure.
    Let us define
    \begin{equation*}
{ \hat{U}}=\left(\cap_{j\in \mathcal{I}}U_{j}\right)\setminus\left(\cup_{j\in \mathcal{I}^{-}}Z_{j}\right).
    \end{equation*}
Notice that for all $u\in{ \hat{U}}$, the spectrum of $H(\cdot)$ is simple and 
\begin{equation*}
    |\langle\phi_{j}(u),H_{2}\phi_{j+1}(u)\rangle|=\|P_{j}(u)H_{2}P_{j+1}(u)\| \neq0,\quad \forall j\in\mathcal{I}^-.
\end{equation*}
Using the superscript $c$ to denote the complement of a set in $U$, we have\begin{equation*}
{ \hat{U}}^{c}=(\cup_{j\in\mathcal{I}}U_{j}^{c})\cup(\cup_{j\in\mathcal{I}^{-}}Z_{j}).\end{equation*} We need to prove that ${ \hat{U}}^{c}$ has empty interior and zero measure.
Since a countable union of subsets of $\mathbb{R}^{m}$ with 
zero measure has 
zero measure and 
empty interior, 
we are left
to prove that for each $j\in\mathcal{I}^-$, the measure of $Z_{j}$ is zero. This can be done by proving that {$u\mapsto P_{j}(u)H_{2}P_{j+1}(u)$} is not identically zero on $U_{j}\cap U_{j+1}$. Let us assume by contradiction that 
\begin{equation}\label{eq:assumption}
    P_{j}(u)H_{2}P_{j+1}(u)=0,\quad\forall u\in U_{j}\cap U_{j+1}.
\end{equation}
Let us take an intersection $\bar{u}=(\bar{u}_{1},\bar{u}_{2})\in U$ between $\lambda_{j}(\cdot)$ and $\lambda_{j+1}(\cdot)$. Since $\bar{u}$ is an isolated { intersection} point in $U$, there exists $\delta>0$ such that 
\begin{equation*}
    [\bar{u}_{1}-\delta,\bar{u}_{1}+\delta]\times[\bar{u}_{2}-\delta,\bar{u}_{2}+\delta]\setminus\{\bar{u}\}\subset U_{j}\cap U_{j+1}.
\end{equation*}
 Define $W=[\bar{u}_{1}-\delta,\bar{u}_{1}+\delta]\times[\bar{u}_{2}-\delta,\bar{u}_{2}+\delta]$.
For $u\in W$, let us denote by $P(u)$ the spectral projection associated with the pair of eigenvalues $\{\lambda_{j}(\cdot),\lambda_{j+1}(\cdot)\}$. 
For $u_{1}\in [\bar{u}_{1}-\delta,\bar{u}_{1}+\delta]$, let us define $\gamma(u_{1},\cdot):[-\delta,\delta]\to W$ by $\gamma(u_{1},t)=(u_{1},\bar{u}_{2}+t)$.
    Denote by $\mathcal{U}(\mathcal{H})$ the group of unitary transformations of $\mathcal{H}$ and let us define the transformation function 
 $V:W\to \mathcal{U}(\mathcal{H})$ such that for all $u_{1}\in [\bar{u}_{1}-\delta,\bar{u}_{1}+\delta]$, $V(u_{1},\cdot)$ satisfies
\begin{multline}\label{eq:ODE}
    \frac{\text{d}}{\text{d}t}V(\gamma(u_{1},t))=\left[\frac{\text{d}}{\text{d}t}P(\gamma(u_{1},t)),P(\gamma(u_{1},t))\right]V(\gamma(u_{1},t)),\quad V(\gamma(u_{1},-\delta))=\text{Id}.
\end{multline}
By classical results on solutions of ODEs,  $V(\cdot,\cdot)$ is well defined and
continuous on $W$. 

We prove next that for each vertical segment $\gamma(u_{1},\cdot)$ that does not intersect $\bar{u}$, and for each $s\in\{j,j+1\}$, the transformed projector $V(\cdot)^{\dag}P_{s}(\cdot)V(\cdot)$ is conserved along the segment.

Let us fix for now $u_{1}\in[\bar{u}_{1}-\delta,\bar{u}_{1})\cup(\bar{u}_{1},\bar{u}_{1}+\delta]$, and set $P(t)=P(\gamma(u_{1},t))$, $V(t)=V(\gamma(u_{1},t))$, {$P_{j}(t)=P_{j}(\gamma(u_{1},t))$, $P_{j+1}(t)=P_{j+1}(\gamma(u_{1},t))$}, $\lambda_{j}(t)=\lambda_{j}(\gamma(u_{1},t))$, $\lambda_{j+1}(t)=\lambda_{j+1}(\gamma(u_{1},t))$, and $H(t)=H(\gamma(u_{1},t))$. 
{
By differentiating with respect to $t$ the equation
\begin{equation*}
    P_{j}(t)P_{j+1}(t)=0,\quad t\in[-\delta,\delta],
\end{equation*}
we obtain that
\begin{equation}
    \dot{P}_{j}(t)P_{j+1}(t)=-P_{j}(t)\dot{P}_{j+1}(t),\qquad t\in[-\delta,\delta].
    \label{eq:projector_1}
\end{equation}
Using the equation
\begin{equation*}
    P_{s}(t)H(t)=H(t)P_{s}(t)=\lambda_{s}(t)P_{s}(t),\qquad s\in\{j,j+1\},
\end{equation*}
and differentiating the equation
\begin{equation*}
    P_{j}(t)H(t)P_{j+1}(t)=0,\qquad t\in[-\delta,\delta],
\end{equation*}
we deduce that
\begin{equation*}
\begin{aligned}
    0={}& 
    \dot{P}_{j}(t)H(t)P_{j+1}(t)+P_{j}(t)\dot{H}(t)P_{j+1}(t)
    +P_{j}(t)H(t)\dot{P}_{j+1}(t)\\
    ={}&\lambda_{j+1}(t)\dot{P}_{j}(t)P_{j+1}(t)
    +P_{j}(t)H_{2}P_{j+1}(t)
    +\lambda_{j}(t)P_{j}(t)\dot{P}_{j+1}(t).
\end{aligned}
\end{equation*}
By equations~\eqref{eq:assumption} and \eqref{eq:projector_1}, and given that $\lambda_{j}(t)\neq\lambda_{j+1}(t)$, we get, for all $t\in[-\delta,\delta]$,
    \begin{equation}\label{eq:dotPjPj+1}
        \dot{P}_{j}(t)P_{j+1}(t)=\frac{1}{\lambda_{j}(t)-\lambda_{j+1}(t)}P_{j}(t)H_{2}P_{j+1}(t)=0=\dot{P}_{j+1}(t)P_{j}(t).
    \end{equation}
    Hence,
\begin{equation}
\begin{aligned}
    \left[\dot{P}(t),P(t)\right]&= \left[\dot{P}_{j}(t)+\dot{P}_{j+1}(t),P_{j}(t)+P_{j+1}(t)\right]\\
    &=\left[\dot{P}_{j}(t),P_{j}(t)\right]+\left[\dot{P}_{j+1}(t),P_{j+1}(t)\right].
    \label{eq:bracket}
\end{aligned}
\end{equation}
Moreover, by differentiating 
\begin{equation*}
    P_{s}(t)^{2}=P_{s}(t),\qquad s\in\{j,j+1\},
\end{equation*}
with respect to $t$, it follows that, for all $t\in[-\delta,\delta]$,
\begin{equation}
\label{eq:projectors}
\dot{P}_{s}(t)P_{s}(t)+P_{s}(t)\dot{P}_{s}(t)=\dot{P}_{s}(t),\qquad s\in\{j,j+1\}.   
\end{equation}
For $s\in\{j,j+1\}$, by left-multiplying 
equation~\eqref{eq:projectors} by $P_{s}(t)$, we can deduce that
\begin{equation*}
    P_{s}(t)\dot{P}_{s}(t)P_{s}(t)+P_{s}(t)\dot{P}_{s}(t)=P_{s}(t)\dot{P}_{s}(t),\qquad t\in[-\delta,\delta],
\end{equation*}
and thus
\begin{equation}
\label{eq:projectors_2}
    P_{s}(t)\dot{P}_{s}(t)P_{s}(t)=0,\qquad t\in[-\delta,\delta].
\end{equation}
By equations~\eqref{eq:ODE}, 
\eqref{eq:dotPjPj+1},
\eqref{eq:bracket}, \eqref{eq:projectors}, 
and~\eqref{eq:projectors_2}, we can then obtain that, for $s\in\{j,j+1\}$ and along each $\gamma(u_{1},\cdot)$, 
\begin{multline*}
    \frac{\text{d}}{\text{d}t}\left(V(t)^{\dag}P_{s}(t)V(t)\right)=-V(t)^{\dag}[\dot{P}(t),P(t)]P_{s}(t)V(t)\\+V(t)^{\dag}\dot{P}_{s}(t)V(t)
    \phantom{=}+V(t)^{\dag}P_s(t)[\dot{P}(t),P(t)]V(t)\\
    =V(t)^{\dag}\Big(-\dot{P}_{s}(t)P_{s}(t)+\dot{P}_{s}(t)-P_{s}(t)\dot{P}_{s}(t)\Big)V(t)=0.
\end{multline*}
Therefore, for $s\in\{j,j+1\}$, $V(\cdot)^{\dag}P_{s}(\cdot)V(\cdot)$ is conserved along each vertical segment $\gamma(u_{1},\cdot)$ with $u_{1}\in[\bar{u}_{1}-\delta,\bar{u}_{1})\cup(\bar{u}_{1},\bar{u}_{1}+\delta]$. For all $u_{1}\in [\bar{u}_{1}-\delta,\bar{u}_{1})\cup(\bar{u}_{1},\bar{u}_{1}+\delta]$ and $u_{2}\in[\bar{u}_{2}-\delta,\bar{u}_{2}+\delta]$, we have that 
\begin{equation*}
    P_{s}(u_{1},u_{2})=V(u_{1},u_{2})P_{s}(u_{1},\bar{u}_{2}-\delta)V(u_{1},u_{2})^{\dagger}.
\end{equation*}
Since $P_{s}(\cdot)$ is analytic on $W\setminus\{\bar{u}\}$ and $V(\cdot)$ is continuous on $W$, by continuous extension, we can obtain that
\begin{equation*}
    P_{s}(u_{1},u_{2})=V(u_{1},u_{2})P_{s}(u_{1},\bar{u}_{2}-\delta)V(u_{1},u_{2})^{\dagger}
\end{equation*}
for every $u\in W\setminus\{\bar{u}\}$.

Consider an analytic path ${ \Gamma}:(-1,1)\to W$ such that ${ \Gamma}(0)=\bar{u}$ and ${ \dot{\Gamma}}(0)$ is a non-zero conical direction at $\bar{u}$ for the intersection between $\lambda_{j}(\cdot)$ and $\lambda_{j+1}(\cdot)$. Then, for $s\in\{j,j+1\}$, we have that
\begin{equation*}
\begin{aligned}
    \lim_{t\rightarrow0^{+}}P_{s}(\Gamma(t))=\lim_{t\rightarrow0^{-}}P_{s}(\Gamma(t))=V(\bar{u}_{1},\bar{u}_{2})P_{s}(\bar{u}_{1},\bar{u}_{2}-\delta)V(\bar{u}_{1},\bar{u}_{2})^{\dagger}.
\end{aligned}
\end{equation*}
However, { by the analytic dependence of the spectrum along $\Gamma$ in a neighborhood of $\Gamma(0)$ (see \cite{rellich1969perturbation}), when passing through the eigenvalue intersection via a conical direction,} we should have that 
\begin{equation*}
    P_{s}(\Gamma(0^{-}))\neq P_{s}(\Gamma(0^{+})).
\end{equation*}
The contradiction is reached, concluding the proof. 
}
\end{proof}

We also recall the following result.

\begin{lemma}[Lemma 3.7 in \cite{liang:hal-04174206}]\label{lemma:non_resonance_finite}\label{lemma:non_resonance_infinite}
Let System~\eqref{system} satisfy \textbf{(H)} or $\textbf{(}\textbf{H}^{\infty}\textbf{)}$. 
Assume that the spectrum of $H(\cdot)$ is weakly conically connected and that the eigenvalue intersections have rationally unrelated germs at every intersection point $\bar{u}\in U$.
 Then there exists a subset  { $\hat{U}$} of  $U$ with zero-measure complement 
 such that 
 if $\sum_{j=1}^{N}\alpha_{j}\lambda_{j}(\bar{u})=0$ with $\bar{u}\in{ \hat{U}}$, $N\in\mathbb{N}^{*}$, $N\le \dim\mathcal{H}$, and $\{\alpha_{1},\dotsc,\alpha_{N}\}\in\mathbb{Q}^{N}$ then $\alpha_{1}=\dots=\alpha_{N}$. 
\end{lemma}

We can now prove Theorem \ref{theorem:main}.

\begin{proof}[Proof of Theorem \ref{theorem:main}] By applying Lemma~\ref{lemma:non_resonance_infinite} we deduce that there exists a subset { $\hat{U}_{1}$} of $U$ which has zero-measure complement such that for all $\bar{u}\in{ \hat{U}_{1}}$ the spectrum is simple and  the eigenvalues $\{\lambda_{j}(\bar{u})\}_{j\in\mathcal{I}}$ are non-resonant (i.e. two spectral gaps $\lambda_{j}(\bar{u})-\lambda_{k}(\bar{u})$ and $\lambda_{r}(\bar{u})-\lambda_{s}(\bar{u})$ are different if $(j,k)\neq(r,s)$ and $j\neq k$, $r\neq s$). By applying Lemma~\ref{lemma:one_input_coupling}, there exists another subset of { $\hat{U}_{2}$} of $U$ which has zero-measure complement such that for all $\bar{u}\in{ \hat{U}_{2}}$ the spectrum is simple and 
     \begin{equation*}
        \langle\phi_{j}(\bar{u}),H_{2}\phi_{j+1}(\bar{u})\rangle\neq 0,\quad \forall j\in\mathcal{I}^-.
    \end{equation*}
Let us define
\begin{equation*}
    { \hat{U}=\hat{U}_{1}\cap\hat{U}_{2}},
\end{equation*}
which has zero-measure complement in $U$. Note that { 
\begin{equation*}
    { \hat{I}^{1}=\pi_1(\hat{U})},
\end{equation*}
where $\pi_{1}$ is the standard projection on the first component}, has zero-measure complement in $I^{1}$ (as defined at the beginning of this section). 
Take { $\bar{u}_{1}\in \hat{I}^{1}$} and consider the single-input system \eqref{one_input_system} for this static value. By definition of { $\hat{I}^{1}$}, there exists $\bar{u}_{2}\in\mathbb{R}$ such that $(\bar{u}_{1},\bar{u}_2)\in U$, the eigenvalues are non-resonant at $(\bar{u}_{1},\bar{u}_2)$ and, for all $j\in\mathcal{I}^{-}$, $\langle\phi_{j}(\bar{u}_1,\bar{u}_2),H_{2}\phi_{j+1}(\bar{u}_1,\bar{u}_2)\rangle\neq 0$. We can deduce from  \cite[Proposition 11]{boscain2015approximate} and \cite[Proposition 15]{boscain2015approximate} that, at every ${ \bar{u}_1\in\hat{I}^{1}}$, the single-input system \eqref{one_input_system} 
is controllable { on} the unitary group when $\text{dim}(\mathcal{H})<\infty$ and is approximately controllable { on} the unitary group when $\text{dim}(\mathcal{H})=\infty.$ 
\end{proof}
{ 
\begin{remark}
Let us define
\[
I^{1}_{\mathrm{NC}} = \left\{ \bar{u}_1 \in I^1 \ \middle| \ 
\begin{matrix}
\text{ the system characterized by }\\H(\bar{u}_1, u_2)\text{ is not controllable }\\ \text{ on the unitary group with } \\
\text{ the control } u_2(\cdot)
\end{matrix}
\right\}.
\]
{ Notice that in the case $\dim \mathcal{H} < \infty$, for $\bar{u}_1\in I^{1}$, the system characterized by $H(\bar{u}_1,u_2)$ is controllable on the unitary group if and only if the Lie algebra generated by $iH_{0}+i\bar{u}_1H_1$ and $iH_2$ is $su(n)$ or $u(n)$.
As a result, if for some $\bar{u}_1$ the system characterized by $H(\bar{u}_1, u_2)$ is controllable on the unitary group with $u_2(\cdot)$, then there exists $\delta > 0$ such that for all $\tilde{u}_1 \in (\bar{u}_1 - \delta, \bar{u}_1 + \delta)$, the system characterized by $H(\tilde{u}_1, u_2)$ remains controllable on the unitary group with $u_2(\cdot)$. Therefore, the set $I^{1}_{\mathrm{NC}}$ is closed, has measure zero, and cannot be dense in $I^{1}$.}

The generalization of this observation to the infinite-dimensional case is not straightforward and could be an interesting topic for further investigation.
\end{remark}}

\section{Application: Enantio-selective excitation in a three-level model for a chiral molecule}\label{s-enantio}
{ Chiral molecules occur as pairs of non-superimposable mirror-image
isomers, known as enantiomers. Enantio-selective excitation of rotational
or ro-vibrational states of chiral molecules as a precursor for chiral
discrimination with electric fields only is presently attracting
significant interest in molecular physics and physical chemistry
\cite{EibenbergerPRL17,PerezAngewandte17,Lee2022,Singh2023}. It is based
on quantum interference in cyclic excitation between three quantum states,
where constructive interference for one enantiomer and destructive
interference for the other leads to selective population of one of the
quantum states \cite{KralPRL01,Leibscher19}.
By analyzing the enantio-selective controllability of asymmetric top
molecules subject to several time-dependent controls, it can be shown that
complete enantio-selective separation can be obtained by a combination of
time-dependent controls with at least five different combinations and
frequencies \cite{Leibscher2022,Pozzoli2022}.

\begin{figure}[!t]
\begin{center}
\includegraphics[width=0.4\columnwidth]{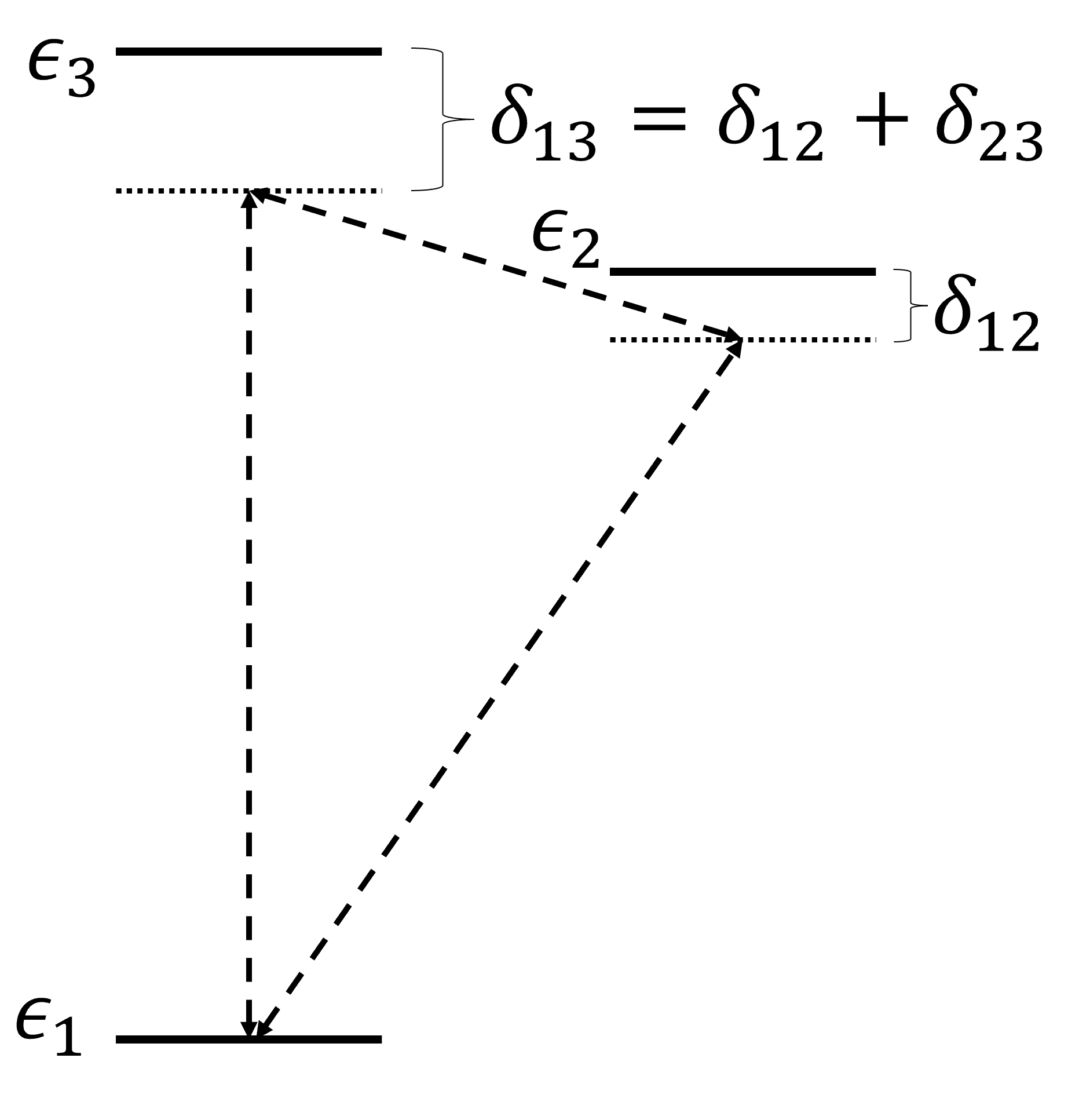}
	\end{center}	
    \caption{
    { Three-level cyclic excitation scheme for enantiomer selective excitation
of ro-vibrational states. The three states $|1\rangle$, $|2\rangle$,
$|3\rangle$ with energies $\epsilon_1$, $\epsilon_2$, and $\epsilon_3$,
respectively represent three ro-vibrational states of a chiral
molecule. The central frequency of the 
}
pulses driving the transition $i \leftrightarrow j$, $i,j=1,2,3$, is detuned from the energy difference $\epsilon_j - \epsilon_i$ by $\delta_{ij}$. Here we choose $\delta_{12}=\delta_{23}$.}  
	\label{fig:scheme_esst}
\end{figure}
In certain cases, i.e. if excited vibrational states are involved in the
process of chiral discrimination, continuous wave IR radiation can be
applied to excited ro-vibrational molecular states. In this case, the
three-level model for enantio-selective excitation \cite{Leibscher19} is
reduced to two (quasi) static inputs and a single time-dependent control,
a scenario for which enantio-selective controllability has not yet been
studied. In this case, the Hamiltonian can be written as
 \begin{eqnarray}\label{eq:H_I}
 	H^{\pm} = \left( \begin{array}{ccc}
 		-\delta_{12} & \pm H_{12} & H_{13} \\
 		\pm H_{12}    & 0 & H_{23}(t) \\
 		H_{13}   & H_{23}(t)  & \delta_{23}
 	\end{array} \right ),
 \end{eqnarray}
 Here, $\delta_{12}$, $\delta_{13}$ and $\delta_{23}$ denote the detunings of the electromagnetic fields from resonance, as depicted in Figure~\ref{fig:scheme_esst}. The detunings are chosen such that $\delta_{13}=\delta_{12}+\delta_{23}$.} The Hamiltonians for the two enantiomers differ in
the sign of one of the three off-diagonal elements, reflecting
the sign difference of one of the dipole moment projections. The off-diagonal matrix elements read $H_{12}=-\mu_{12} {\cal E}_{12}$,  $H_{13}=-\mu_{13} {\cal E}_{13}$, and $H_{23}(t)=-\mu_{23} {\cal E}_{23}(t)$, with $\mu_{ij}$ being the (transition) dipole moments and ${\cal E}_{12}$ and ${\cal E}_{13}$ the constant field amplitudes, while the field amplitude ${\cal E}_{12}(t)$
 is time-dependent.
 If the corresponding Schr\"odinger equation is simultaneously controllable for $H^+$ and $H^-$, complete enantio-selective state transfer is possible in the setup described above.
{\begin{remark}
    Notice that in equation~\eqref{eq:H_I}, the trace of $H^{\pm}$ is not necessarily zero. However, we can always choose $\delta_{12}=\delta_{23}$ in order for the trace to be zero. In the following, we will only study the case where $H^{\pm}$ is traceless. 
 \end{remark}}
 We investigate the controllability of the system consisting of the three-level drift Hamiltonian 
\begin{equation*}
    H_{0}=\begin{pmatrix}
        E_1 & 0 & 0\\
        0 & E_2 & 0\\
        0 & 0 & E_3
    \end{pmatrix},
\end{equation*}
where $E_1< E_2 < E_3$ and $E_1+E_2+E_3=0$, and set
\begin{equation*}
      H_{u}=\begin{pmatrix}
        0 & 1 & 0\\
        1 & 0 & 0\\
        0 & 0 & 0
    \end{pmatrix},\  H_{v}=\begin{pmatrix}
        0 & 0 & 1\\
        0 & 0 & 0\\
        1 & 0 & 0
    \end{pmatrix}\  H_{w}=\begin{pmatrix}
        0 & 0 & 0\\
        0 & 0 & 1\\
        0 & 1 & 0
    \end{pmatrix}.
\end{equation*}
Notice that $iH_0,iH_{u},iH_{v},iH_{w}\in {su}(3)$. For $\uu=(u,v,w)\in\mathbb{R}^{3}$, define 
  \begin{equation*}
  \begin{aligned}
      H^{+}(\uu)&=H_0+uH_u+vH_v+w H_{w},\\ H^{-}(\uu)&=H_0-uH_u+vH_v+w H_{w}.
    \end{aligned}
  \end{equation*}
For $(\bar{u},\bar{v})\in\mathbb{R}^2$, let us consider the 
two systems 
\begin{equation}
    \begin{aligned}
    \label{eq:simultaneous}
        i\frac{\text{d}}{\text{d}t}\psi^{+}(t)=H^{+}(\bar{u},\bar{v},w(t))\psi^{+}(t),\\
        i\frac{\text{d}}{\text{d}t}\psi^{-}(t)=H^{-}(\bar{u},\bar{v},w(t))\psi^{-}(t).
    \end{aligned}
\end{equation}
\begin{prop}
\label{prop:individual_controllability}
    Assume that the control $w(\cdot)$ takes values in $\mathbb{R}$. Then for almost every $(\bar{u},\bar{v})\in\mathbb{R}^{2}$, each system in equation \eqref{eq:simultaneous} is controllable. 
\end{prop}
\begin{proof}
{ First, notice that the system characterized by $H^{+}(\bar{u},\bar{v},w(t))$ is controllable if and only if $\text{Lie}(iH^{+}(\bar{u},\bar{v},0),iH_{w})=su(3)$. Using this Lie group characterization, we can deduce that the subset
\begin{equation*}
    U_{C}=\left\{(\bar{u},\bar{v})\in\mathbb{R}^2\middle|~\text{Lie}(iH^{+}(\bar{u},\bar{v}),iH_w)=su(3)\right\}
\end{equation*}
is open. Its complement in $\mathbb{R}
^2$ is  thus closed and, in particular, measurable. L}et us fix $\bar{v}\in\mathbb{R}$ and define $H_{\bar{v}}(\cdot,\cdot)$ as
    \begin{equation*}
        H_{\bar{v}}(u,w)=H^{+}(u,\bar{v},w).
    \end{equation*}
    Since $E_1<E_3$, we can introduce the angle
    \begin{equation}
    \label{eq:def_theta}
        \theta=\arctan\left(\frac{2\bar{v}}{E_3-E_1}\right),
    \end{equation}
    and the unitary transformation 
    \begin{equation*}
        P(\theta)=\begin{pmatrix}
            \cos\left(\frac{\theta}{2}\right) & 0 & -\sin\left(\frac{\theta}{2}\right)\\
            0 & 1 & 0\\
            \sin\left(\frac{\theta}{2}\right) & 0 & \cos\left(\frac{\theta}{2}\right)
        \end{pmatrix}.
    \end{equation*}
    Define the transformed Hamiltonian 
    \begin{equation*}
        \Tilde{H}_{\bar{v}}(\tilde{u},\tilde{w})=P(\theta)H_{\bar{v}}(u,w)P^{\top}(\theta)=\begin{pmatrix}
            \Tilde{E}_1 & \tilde{u} & 0\\
            \tilde{u} & E_2 & \tilde{w}\\
            0 & \tilde{w} & \Tilde{E}_3
        \end{pmatrix},
    \end{equation*}
    where
    \begin{equation}\label{def:E1_E2}
        \begin{aligned}
            \tilde{u}&=\cos\left(\frac{\theta}{2}\right)u-\sin\left(\frac{\theta}{2}\right)w,\\ \tilde{w}&=\sin\left(\frac{\theta}{2}\right)u+\cos\left(\frac{\theta}{2}\right)w,\\
            \tilde{E}_1&=\frac{E_1+E_2-\sqrt{(E_1-E_2)^2+4\bar{v}^2}}{2},\\ \tilde{E_3}&=\frac{E_1+E_2+\sqrt{(E_1-E_2)^2+4\bar{v}^2}}{2}.
        \end{aligned}
    \end{equation}
    Denote by $\tilde{\lambda}_{1}(\tilde{u},\tilde{w})\leq \tilde{\lambda}_{2}(\tilde{u},\tilde{w})\leq\tilde{\lambda}_{3}(\tilde{u},\tilde{w})$ the eigenvalues of $\tilde{H}_{\bar{v}}(\tilde{u},\tilde{w})$ and by  $\lambda_1(u,w)\leq\lambda_{2}(u,v)\leq\lambda_{3}(u,v)$ those of $H_{\bar{v}}(u,w)$. Evidently, the spectrum of $\tilde{H}_{\bar{v}}(\tilde{u},\tilde{w})$ is conically connected with conical intersections between $\tilde{\lambda}_1$ and $\tilde{\lambda}_2$ at
    \begin{equation*}
    (\tilde{u},\tilde{w})=\left(\pm\sqrt{(\tilde{E}_3-\tilde{E}_1)(\tilde{E}_3-E_2)},0\right)
    \end{equation*}
    and conical intersections between $\tilde{\lambda}_2$ and $\tilde{\lambda}_3$ at 
    \begin{equation*}
        (\tilde{u},\tilde{w})=\left(0,\pm\sqrt{(\tilde{E}_3-\tilde{
    E}_1)(E_2-\tilde{E}_1)}\right).
    \end{equation*}
    Then we can deduce that the spectrum of $H_{\bar{v}}(u,w)$ is conically connected with conical intersections between $\lambda_1$ and $\lambda_2$ at
    \begin{equation}
    \label{eq:conical_intersection_1}
        (u,w)=\pm\sqrt{(\tilde{E}_3-\tilde{E}_1)(\tilde{E}_3-E_2)}\left(\cos\left(\frac{\theta}{2}\right),-\sin\left(\frac{\theta}{2}\right)\right),
    \end{equation}
    and conical intersections between $\lambda_2$ and $\lambda_3$ at
    \begin{equation}
    \label{eq:conical_intersection_2}
          (u,w)=\pm\sqrt{(\tilde{E}_3-\tilde{
    E}_1)(E_2-\tilde{E}_1)}\left(\sin\left(\frac{\theta}{2}\right),\cos\left(\frac{\theta}{2}\right)\right).
    \end{equation}
See Figure \ref{fig:example_2} for an illustrative example when $(E_1,E_2,E_3)=(-1.5,0.5,1)$ and $\bar{v}=3$.\\
    \begin{figure}[!t]
\centering
\includegraphics[width=0.6\columnwidth]{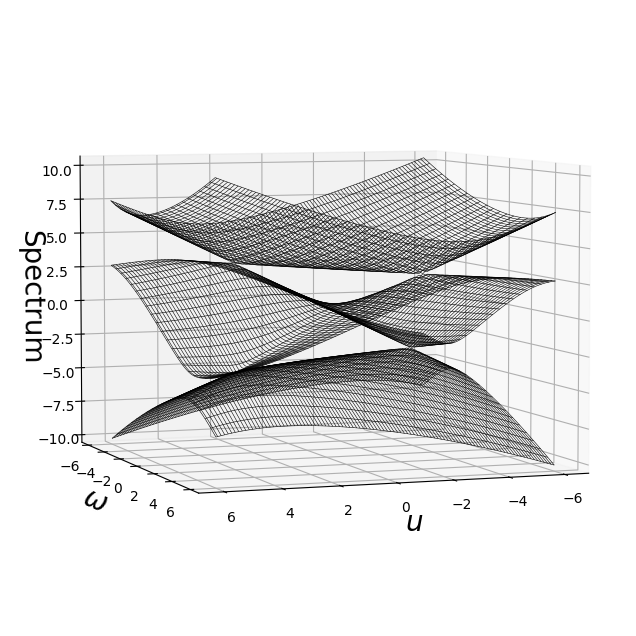}
    \caption{Case $(E_1,E_2,E_3)=(-1.5,0.5,1)$ and $\bar{v}=3$: the spectrum of $H^{+}(u,\bar{v},w)$ as a function of $(u,w)$ is conically connected}
    \label{fig:example_2}
\end{figure}
     By Theorem 4.1, we deduce that for almost every $\bar{u}\in\mathbb{R}$, the system characterized by $H_{\bar{v}}(\bar{u},w)$ is controllable with $w(\cdot)$, which is true for all $\bar{v}\in\mathbb{R}$. { 
Then we conclude by Fubini's theorem that the complement of $U_C$ in $\mathbb{R}^2$ has measure zero.
} Therefore, for almost every $(\bar{u},\bar{v})\in\mathbb{R}^2$, the system characterized by $H^{+}(\bar{u},\bar{v},w)$ in equation \eqref{eq:simultaneous} is controllable with $w(\cdot)$. 
     
     A similar reasoning shows that the system characterized by $H^{-}(\bar{u},\bar{v},w)$ in equation \eqref{eq:simultaneous} is also controllable.
\end{proof}

 For $\uu=(u,v,w)\in\mathbb{R}^{3}$, denote by $\lambda_1^{+}(\uu)\leq\lambda_2^{+}(\uu)\leq\lambda_3^{+}(\uu)$ the eigenvalues and by $\phi^{+}_{1}(\uu),\phi^{+}_2(\uu),\phi^{+}_3(\uu)$ { an orthonormal triple of corresponding eigenstates} of $H^{+}(\uu)$ and by $\lambda_{1}^{-}(\uu)\leq\lambda_{2}^{-}(\uu)\leq\lambda_3^{-}(\uu)$ and $\phi_{1}^{-}(\uu),\phi_2^{-}(\uu),\phi_3^{-}(\uu)$ those of $H^{-}(\uu)$. 
\begin{lemma}
\label{lemma:non_resonance}
    There exists a subset $U$ of $\mathbb{R}^3$ with zero-measure complement such that for all $\uu=(u,v,w)\in{ U}$,
   \begin{equation}
   \begin{aligned}
   \label{eq:non_resonance}
    \lambda^{+}_{2}(\uu)-\lambda^{+}_{1}(\uu)&\neq0,\quad\left\langle\phi_{1}^{+}(\uu),H_{w}
    \phi_2^{+}(\uu)\right\rangle\neq0,\\
    \lambda^{+}_2(\uu)-\lambda^{+}_1(\uu)&\notin\left\{\lambda_{k}^{-}(\uu)-\lambda_{j}^{-}(\uu)\mid 1\leq j<k\leq 3\right\}.
    \end{aligned}
\end{equation}
    \begin{proof}
    { Define $U$ as the set of $\uu=(u,v,w)\in\mathbb{R}^{3}$ such that all conditions in equation~\eqref{eq:non_resonance} are verified at $\uu$.} Let us fix $\bar{v}\in\mathbb{R}^{*}$. As noticed in  the proof of Proposition~\ref{prop:individual_controllability}, the spectrum of $H(u,\bar{v},w)$ as a function of $(u,w)\in\mathbb{R}^2$ is conically connected. Then, by Lemma~\ref{lemma:one_input_coupling}, there exists a dense subset $V_1(\bar{v})$ of $\mathbb{R}^2$ such that for all 
    $(u,w)\in V_1(\bar{v})$,
    \begin{equation*}
    \begin{aligned}
        \lambda_{2}^{+}(u,\bar{v},w)-\lambda_{1}^{+}(u,\bar{v},w)&\neq 0,\\
        \left\langle\phi_1^{+}(u,\bar{v},w),H_w\phi_2^{+}(u,\bar{v},w)\right\rangle&\neq 0.
    \end{aligned}
    \end{equation*}
Take a conical intersection $(\bar{u},\bar{w})$ between $\lambda^{+}_1(\cdot,\bar{v},\cdot)$ and $\lambda^{+}_2(\cdot,\bar{v},\cdot)$ as given in equation~\eqref{eq:conical_intersection_1}, with $\theta\in(-\pi,\pi)$ defined in equation \eqref{eq:def_theta} and $(\tilde{E}_1,\tilde{E}_3)$ defined in equation \eqref{def:E1_E2}. Since $\bar{v}\neq 0$, it can be checked that $(\bar{u},\bar{w})$ cannot be an eigenvalue intersection of the spectrum of $H^{-}(\cdot,\bar{v},\cdot)$, which means that
\begin{equation*}
    \lambda^{-}_{k}(\bar{u},\bar{v},\bar{w})-\lambda^{-}_{j}(\bar{u},\bar{v},\bar{w})\neq 0,\quad\mbox{ for } 1\leq j<k\leq 3.
\end{equation*}
Then, by the analyticity of the spectra of $H^{+}(\cdot,\bar{v},\cdot)$ and $H^{-}(\cdot,\bar{v},\cdot)$, we can deduce that there exists a subset $V_2(\bar{v})$ of $\mathbb{R}^2$ with zero-measure complement such that for all $(u,w)\in V_2(\bar{v})$
\begin{multline*}
    \lambda^{+}_{2}(u,\bar{v},w)-\lambda^{+}_{1}(u,\bar{v},w) \\ \notin\left\{\lambda_{k}^{-}(u,\bar{v},w)-\lambda_{j}^{-}(u,\bar{v},w)\mid 1\leq j<k\leq 3\right\}.
\end{multline*}
{ 
We observe that $U$ is a measurable subset of $\mathbb{R}^{3}$ and that
\begin{equation*}
    \cup_{\bar{v}\in\mathbb{R}^{*}}\{\bar{v}\}\times\left(V_1(\bar{v})\cap V_2(\bar{v})\right)\subset U,
\end{equation*} 
Finally, by Fubini's theorem, we conclude that the complement of $U$ in $\mathbb{R}^3$ has measure zero.}
\end{proof}
\end{lemma}
\begin{lemma}
\label{lemma:no_isomorphism_1}
    There exists a subset $U$ of $\mathbb{R}^{3}$ with zero-measure complement such that, for all $\uu\in U$, there exists no Lie algebra isomorphism $f: {su}(3)\rightarrow {su}(3)$ such that $f\left(iH^{+}(\uu)\right)=iH^{-}(\uu)$ and $f\left(iH_w\right)=iH_w$.
\end{lemma}
\begin{proof}
    By Lemma \ref{lemma:non_resonance}, there exists a subset $U$ of $\mathbb{R}^3$ with zero-measure complement such that the conditions in \eqref{eq:non_resonance} are true for all $\uu\in U$. Notice that for $p\in\mathbb{N}$,
    \begin{equation*}
        \begin{aligned}
            \text{ad}^{2p}_{iH^{+}(\uu)}(iH_w)&=\sum_{j,k=1}^3i\big(-(\lambda^{+}_{j}(\uu)-\lambda^{+}_{k}(\uu))^2\big)^{p}\left\langle \phi^{+}_{j}(\uu),H_w\phi^{+}_{k}(\uu)\right\rangle
            \phi^{+}_j(\uu)\left(\phi^{+}_k(\uu)\right)^{\dagger}           
            ,\\
            \text{ad}^{2p}_{iH^{-}(\uu)}(iH_w)&=\sum_{j,k=1}^3i\big(-(\lambda^{-}_{j}(\uu)-\lambda^{-}_{k}(\uu))^2\big)^{p}\left\langle \phi^{-}_{j}(\uu),H_w\phi^{-}_{k}(\uu)\right\rangle
\phi^{ -}_j(\uu)\left(\phi^{ -}_k(\uu)\right)^{\dagger}.            
        \end{aligned}
    \end{equation*}
Take a polynomial $P(X)=\sum_{p=0}^{m}a_pX^{p}$ such that $P(x)=0$ if $x\in\{0,-(\lambda_{1}^{-}(\uu)-\lambda^{-}_{2}(\uu))^2,-(\lambda^{-}_1(\uu)-\lambda^{-}_{3}(\uu))^2,-(\lambda^{-}_2(\uu)-\lambda^{-}_{3}(\uu))^2\}$ and $P(x)\neq 0$ otherwise. Then  
\begin{equation*}
    \sum_{p=0}^{m}a_{p}\text{ad}^{2p}_{iH^{-}(\uu)}(iH_w)=0,
\end{equation*}
and, according to equation \eqref{eq:non_resonance}, 
\begin{equation*}
    \sum_{p=0}^{m}a_{p}\text{ad}^{2p}_{iH^{+}(\uu)}(iH_w)\neq 0.
\end{equation*}
Assume by contradiction that there exists a Lie algebra isomorphism $f: {su}(3)\rightarrow {su}(3)$ such that $f\left(iH^{+}(\uu)\right)=iH^{-}(\uu)$ and $f(iH_w)=iH_w$. Then we  have
\begin{align*}
    f\left(\sum_{p=0}^{m}a_{p}\text{ad}^{2p}_{iH^{+}(\uu)}(iH_w)\right)&=\sum_{p=0}^{m}a_{p}\text{ad}^{2p}_{f(iH^{+}(\uu))}(f(iH_w))\\
    &=\sum_{p=0}^{m}a_{p}\text{ad}^{2p}_{iH^{-}(\uu)}(iH_w)=0.
\end{align*}
The contradiction is reached since $f^{-1}(0)=0$. 
\end{proof}
\begin{lemma}
\label{lemma:no_isomorphism_2}
    There exists a subset $V$ of $\mathbb{R}^2$ with zero-measure complement such that for each $(u,v)\in V$ there exists no Lie algebra isomorphism $f: {su}(3)\rightarrow {su}(3)$ such that $f(iH^{+}(u,v,0))=iH^{-}(u,v,0)$ and $f(iH_w)=iH_w$.
\end{lemma}
\begin{proof}
    Assume by contradiction that
    there exists $V^c\subset \mathbb{R}^2$ with  non-zero measure such that for each $({u},{v})\in V^{c}$  there exists a Lie algebra isomorphism $f$ satisfying $f(iH^{+}({u},{v},0))=iH^{-}(u,v,0)$ and $f(iH_w)=iH_w$. Then by the linearity of $f$, for all $w\in\mathbb{R}$, we have $f(iH^{+}(u,v,w))=iH^{-}(u,v,w)$ and $f(iH_w)=iH_w$. Then such isomorphism exists for all 
    \begin{equation*}
         \uu=(u,v,w)\in V^{c}\times \mathbb{R},
    \end{equation*}
and $V^{c}\times\mathbb{R}$ is a subset of $\mathbb{R}^3$ with non-zero measure. This is in contradiction with Lemma \ref{lemma:no_isomorphism_1}, concluding the proof.
\end{proof}
\begin{prop}
    For almost every $(\bar{u},\bar{v})\in\mathbb{R}^2$, the two systems in equation \eqref{eq:simultaneous} are simultaneously controllable with control $w(\cdot)$.
\end{prop}
\begin{proof}
    By Proposition \ref{prop:individual_controllability}, there exists a subset $V_1$ of $\mathbb{R}^2$ with zero-measure complement, such that for all $(\bar{u},\bar{v})\in V_1$, each system in equation \eqref{eq:simultaneous} is controllable with control $w(\cdot)$. Moreover by Lemma~\ref{lemma:no_isomorphism_2}, there exists a subset $V_2$ of $\mathbb{R}^{2}$ with zero-measure complement such that for all $(\bar{u},\bar{v})\in { V_2}$, there exists no Lie algebra isomorphism $f: {su}(3)\rightarrow {su}(3)$ such that $f(iH^{+}(\bar{u},\bar{v},0))=iH^{-}(\bar{u},\bar{v},0)$ and $f(iH_w)=iH_w$. 
    Since $V=V_1\cap V_2$ is a dense subset of $\mathbb{R}^2$ with zero measure complement, the conclusion follows from Lemma~3 in \cite{MR3318240}. 
\end{proof}
{Applying the proven result to the scenario of enantio-selective controllability subject to two continuous wave IR fields and one time-dependent microwave pulse, modeled in the rotating wave approximation, we find controllability for $\delta_{12}=\delta_{23}$. For almost all static values $\bar{u}$ and $\bar{v}$ in \eqref{eq:simultaneous} (i.e. $(\bar{u},\bar{v})$ in a subset of $\mathbb{R}^2$ with zero-measure complement), with $\bar{u}$ and
$\bar{v}$ denoting Rabi frequencies $\Omega_{12}=\mu_{12} {\cal E}_{12}$
and $\Omega_{13}=\mu_{13} {\cal E}_{13}$, 
the time-dependent control $w(t)$ yields  enantio-selective
controllability.}

\section{Application:  Driven Jaynes--Cummings Hamiltonian}
\label{s-jaynes}
{The Jaynes-Cummings Hamiltonian is a paradigmatic model in quantum optics. It describes a two-level system coupled to a quantum harmonic oscillator where the coupling is sufficiently weak to allow for the rotating wave approximation. Experimental realizations include atoms or molecules interacting with a quantized mode of the electromagnetic field in a cavity or internal states of trapped ions which are coupled to the quantized motion in the ion trap. Without external drive, the model is analytically solvable.}

Here we consider,  in the Hilbert space  $\mathcal{H}=L^{2}(\mathbb{R})\otimes\mathbb{C}^{2}$, the driven Jaynes--Cummings system
\begin{equation}
    \label{eq:system_3}
    i\dot{\psi}(t)=H_{\rm JC}(u(t))\psi(t),
\end{equation}
where the control $u(\cdot)$ (typically a classical electromagnetic field) is  piecewise constant and takes values in $\mathbb{R}$. 
The driven Jaynes--Cummings Hamiltonian 
is defined as
{ 
\begin{equation*}
\begin{aligned}
    H_{\rm JC}(u)=&\omega \left(a^{\dagger}a+\frac{1}{2}\right)\otimes\mathbb{I}_{\mathbb{C}^2}+\frac{\Omega}{2}\mathbb{I}_{L^{2}(\mathbb{R})}\otimes\sigma_{z}\\&+g\left(a\otimes\sigma_{+}+a^{\dagger}\otimes\sigma_{-}\right)+u\left(a^{\dagger}+a\right)\otimes\mathbb{I}_{\mathbb{C}^2},
\end{aligned}
\end{equation*}
where $\omega,\Omega>0$ and $g\in\mathbb{R}$ are scalar parameters of the system. Here, $\mathbb{I}_{L^{2}(\mathbb{R})}$ and $\mathbb{I}_{\mathbb{C}^2}$ denote the identity operators on the Hilbert spaces $L^{2}(\mathbb{R})$ and $\mathbb{C}^2$, respectively.}
 $a^{\dagger}$ and $a$ represent the creation and annihilation operators for the harmonic oscillator, that is,
\begin{equation}
\label{eq:creation_annihilation}
    a^{\dagger}=\frac{1}{\sqrt{2}}\left(x-\frac{\partial}{\partial x}\right),\quad  a=\frac{1}{\sqrt{2}}\left(x+\frac{\partial}{\partial x}\right).
\end{equation}
$\sigma_{z}$ is the Pauli matrix $(\begin{smallmatrix}1&0\\0&-1\end{smallmatrix})$,
whose eigenvectors are denoted by $e_{1},e_{-1}$, and $\sigma_{-}$ and $\sigma_{+}$ are given by
\begin{equation*}
    \sigma_{-}=\ket{e_{-1}}\bra{e_{1}}=\begin{pmatrix}
        0 & 0\\
        1 & 0
    \end{pmatrix},\quad\sigma_{+}=\ket{e_{1}}\bra{e_{-1}}=\begin{pmatrix}
        0 & 1\\
        0 & 0
    \end{pmatrix}.
\end{equation*}
{ In the rest of this section, following the classical physical notation, we omit the evident tensor products and identity operators, which leads to the expression:}
\begin{equation*}
    H_{\rm JC}(u)=\omega \left(a^{\dagger}a+\frac{1}{2}\right)+\frac{\Omega}{2}\sigma_{z}+g(a\sigma_{+}+a^{\dagger}\sigma_{-})+u(a^{\dagger}+a).
\end{equation*}
We can deduce from the results of the previous sections the following result on the controllability of the driven Jaynes--Cummings system.
\begin{theorem}
\label{theorem:Controllability_JC}
    For almost every $\Omega,\omega>0$ and $g\in\mathbb{R}$, the driven Jaynes--Cummings system \eqref{eq:system_3} is approximately controllable { on} the unitary group.
\end{theorem}
This result, under a slightly different assumption, has already been proven in \cite{10.1063/1.5023587}. Here we propose an alternative proof relying on weak conical connectedness.

{ Before proving Theorem \ref{theorem:Controllability_JC}, let us first consider the two-input Hamiltonian}
\begin{equation}
\label{eq:Hamiltonian_two_inputs}
\begin{aligned}
H(u_{1},u_{2})=&H_{0}+u_{1}H_{1}+u_{2}H_{2}\\
    =&\omega \left(a^{\dagger}a+\frac{1}{2}\right)+\frac{\Omega}{2}\sigma_{z}+u_{1}(a\sigma_{+}+a^{\dagger}\sigma_{-})+u_{2}{(a^{\dagger}+a)}.
\end{aligned}
\end{equation}
The associated bilinear Schr\"odinger equation is given by
\begin{equation}
    \label{eq:system_2}
    i\dot{\psi}(t)=H(u(t))\psi(t)=H(u_{1}(t),u_{2}(t))\psi(t).
\end{equation}
Here $u(\cdot)$ is a piecewise constant function that takes its values in $U=\mathbb{R}^{2}$. {It is important to notice that $H(g,u)=H_{\rm JC}(u)$. The parameter $g$ can thus be considered as a static control for system \eqref{eq:system_2}. In the following, we will first prove that the spectrum of $H(u_1,u_2)$ is weakly conically connected and has rationally unrelated germs at each intersection point. Then we will apply Theorem \ref{theorem:main} to conclude the proof of Theorem \ref{theorem:Controllability_JC}.}

Let us notice that $H_{1},H_{2}$ are infinitesimally small with respect to $H_{0}$ in the sense of Kato \cite{kato1966perturbation}. The eigenfunctions of $a^{\dagger}a$ are the Hermite functions
\begin{equation*}
    \ket{n}=\Phi_{n}(\cdot)=\left(x\mapsto \frac{1}{\sqrt{2^{n}n!\sqrt{\pi}}}h_{n}(x)e^{-\frac{x^{2}}{2}}\right),\quad n\in\mathbb{N},
\end{equation*}
where $h_{n}$ denotes the $n$th Hermite polynomial. An analysis of the spectrum of $H_{0}$ can be conducted as in \cite{10.1063/1.5023587}. The eigenvectors of $H_{0}$ can be obtained by tensorization, namely, 
\begin{equation}
\label{eq:eig_H_0}
\begin{aligned}
     H_{0}\ket{n}\otimes e_{1}&=E_{(n,1)}^{0}\ket{n}\otimes e_{1}=\Big(\omega(n+\frac{1}{2})+\frac{\Omega}{2}\Big)\ket{n}\otimes e_{1},\\
     H_{0}\ket{n}\otimes e_{-1}&=E_{(n,-1)}^{0}\ket{n}\otimes e_{-1}=\Big(\omega(n+\frac{1}{2})-\frac{\Omega}{2}\Big)\ket{n}\otimes e_{-1}.
\end{aligned}
\end{equation}
\begin{lemma}
\label{lemma:non-degenerate}
    Assume that $\Omega/\omega$ is irrational. Then for every $u_{1}\in\mathbb{R}$ and $u_{2}\in\mathbb{R}^{*}$, all eigenvalues of $H(u_{1},u_{2})$ are non-degenerate.
\end{lemma}
\begin{proof}
    Consider $u_{1}\in\mathbb{R}$ and $u_{2}\in\mathbb{R}^{*}$. Define 
    {$\delta=\frac{u_2}{\omega}$} and the translated creation and annihilation operators
    \begin{equation*}
        \tilde{a}=a+\delta,\quad \tilde{a}^{\dagger}=a^{\dagger}+\delta.
    \end{equation*}
    The Hamiltonian $H$ { given by equation~\eqref{eq:Hamiltonian_two_inputs}} can be then written in the translated form
    \begin{equation}
     \begin{aligned}
         H(u_{1},u_{2})&=\omega\left((\tilde{a}-\delta)^{\dagger}(\tilde a-\delta)+\frac{1}{2}\right)+\frac{\Omega}{2}\sigma_{z}\\
         &\phantom{=}+u_{1}\left((\tilde{a}^{\dagger}-\delta)\sigma_{-}+(\tilde{a}-\delta)\sigma_{+}\right)+u_{2}(\tilde{a}^{\dagger}+\tilde{a})-2u_2\delta\\
         &=\omega\left(\tilde{a}^{\dagger}\tilde{a}+\frac{1}{2}+\delta^{2}\right)+(u_2-\delta\omega)(\tilde{a}^{\dagger}+\tilde{a})\\&\phantom{=}+\frac{\Omega}{2}\sigma_{z}\phantom{=}+u_{1}(\tilde{a}^{\dagger}\sigma_{-}+\tilde{a}\sigma_{+})
         -u_{1}\delta \sigma_{x}-2u_2\delta\\ &=\omega\left(\tilde{a}^{\dagger}\tilde{a}+\frac{1}{2}\right)+\frac{\Omega}{2}\sigma_{z}+u_{1}(\tilde{a}^{\dagger}\sigma_{-}+\tilde{a}\sigma_{+})\\
         &\phantom{=}-u_{1}\delta\sigma_{x}+(\omega\delta^2-2u_{2}\delta).
     \end{aligned}   
    \end{equation}
    Here { $\sigma_{x}$} denotes the Pauli matrix $(\begin{smallmatrix}0&1\\1&0\end{smallmatrix})$. 
    Notice that the eigenfunctions of $\tilde{a}^{\dagger}\tilde{a}$ are given by the translated Hermite functions
    \begin{equation}
        \ket{\tilde{n}}=\hat{D}^{\dagger}(\delta)\ket{n},
    \end{equation}
    where { $\hat{D}(\delta)=\exp(\delta a^{\dagger}-\delta a)$ is the displacement operator associated with $\delta$}.
Define another Hamiltonian 
\begin{equation}
    \Tilde{H}(\tilde{u}_{1},\tilde{u}_{2})=\omega\left(\tilde{a}\tilde{a}^{\dagger}+\frac{1}{2}\right)+\frac{\Omega}{2}\sigma_{z}+\tilde{u}_{1}(\tilde{a}^{\dagger}\sigma_{-}+\tilde{a}\sigma_{+})+\tilde{u}_{2} \sigma_{x}.
\end{equation}
Notice that the eigenvalues of $\tilde{H}(\tilde{u}_{1},\tilde{u}_{2})$ are non-degenerate if $\tilde{u}_{1}\neq 0$ and $\tilde{u}_{2}\neq 0$ \cite{liang:hal-04174206} and that 
\begin{equation*}
    H(u_{1},u_{2})=\tilde{H}(u_{1},-u_{1}\delta)+\omega\delta^{2}-{2u_{2}\delta}.
\end{equation*}
When $u_{1}=0$, we have $H(u_{1},u_{2})=\tilde{H}(0,0)+\omega\delta^{2}-{2u_2\delta}$. Since $\Omega/\omega$ is irrational, the eigenvalues of $H(u_{1},u_{2})$ are non-degenerate. When $u_{1}\neq 0$, $H(u_{1},u_{2})=\tilde{H}(u_{1},-u_{1}\delta)+\omega\delta^{2}-{ 2u_2\delta}$ and its eigenvalues are non-degenerate since $u_{1}\neq 0$ and $u_{1}\delta\neq 0$. We can deduce that the eigenvalues of $H(u_{1},u_{2})$ with $u_{2}\neq 0$ are non-degenerate.
\end{proof}
\begin{prop}\label{prop:weakly-conical}
    For almost every $\Omega, \omega>0$, the spectrum of $H(\cdot,\cdot)$ is weakly conically connected and the eigenvalue intersections are non-overlapping.
\end{prop}
\begin{proof}
    Consider the Hamiltonian $h(g):=H_{0}+gH_{1}$, $g\in\mathbb{R}$, and define $\Delta=\Omega-\omega$. Set, moreover, $\delta$ to be the symbol $+$ if $\Delta\geq 0$ and the symbol $-$ if $\Delta<0$.
    According to the spectrum analysis conducted in \cite{10.1063/1.5023587}, 
    the eigenvalues of $h(g)$ are given by
    \begin{equation*}
    \begin{aligned}
        E_{n,\nu}(g)&=\omega(n+1)+\nu\frac{1}{2}\sqrt{\Delta^{2}+4g^{2}(n+1)},\qquad\mbox{ for }n\in\mathbb{N},\nu\in\{+,-\},\\
        E_{-1,\delta}(g)&=\frac{\Delta}{2}.
    \end{aligned}
    \end{equation*}
Notice that, for $n\in\mathbb{N}$ and $\nu\in\{+,-\}$, the asymptotic slope of $g\mapsto E_{n,\nu}(g)$  for $g\gg 1$
    is $\nu\sqrt{n+1}$.

    To relate these eigenvalues to those of $H_{0}$ in \eqref{eq:eig_H_0}, we  distinguish between three situations. At $g=0$, we  have
    \begin{equation*}
        \begin{aligned}
            \begin{cases}E_{(n,+)}(0)=E_{(n,1)}^{0}&\mbox{for $n\ge 0$,}\\
            E_{(n,-)}(0)=E_{(n+1,-1)}^{0}&\mbox{for $n\ge -1$,}
            \end{cases}\quad&\text{if }\Delta>0,\\
            \begin{cases}E_{(n,+)}(0)=E_{(n+1,-1)}^{0}&\mbox{for $n\ge -1$,}\\ E_{(n,-)}(0)=E_{(n,1)}^{0}&\mbox{for $n\ge 0$,}\end{cases}\quad&\text{if }\Delta<0,\\
            E_{(n,\nu)}(0)=E_{(n,1)}^{0}=E_{(n+1,-1)}^{0}\quad\mbox{for }{ \nu}\in\{+,-\}\quad&\text{if }\Delta=0,
        \end{aligned}
    \end{equation*}
    where, in the last line, $n\ge -1$ if $\nu=+$ and $n\ge 0$ if $\nu=-$. 
    Take $(n,\nu)\in(\mathbb{N}\times\{+,-\})\cup\{(-1,\delta)\}$. 

\begin{figure}[!t]
\centering
     \begin{subfigure}[b]
     {0.45\columnwidth}
         \centering
         \includegraphics[width=\columnwidth]{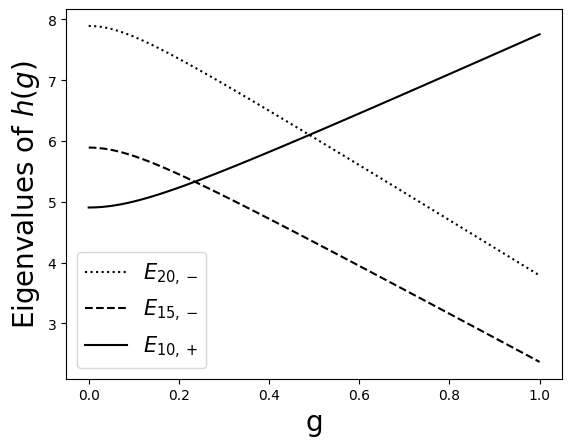}
         \caption{Case $(n,\nu)=(10,+)$}
         \label{fig:y equals x}
     \end{subfigure}
     \hfill
     \begin{subfigure}[b]{0.45\columnwidth}
         \centering
         \includegraphics[width=\columnwidth]{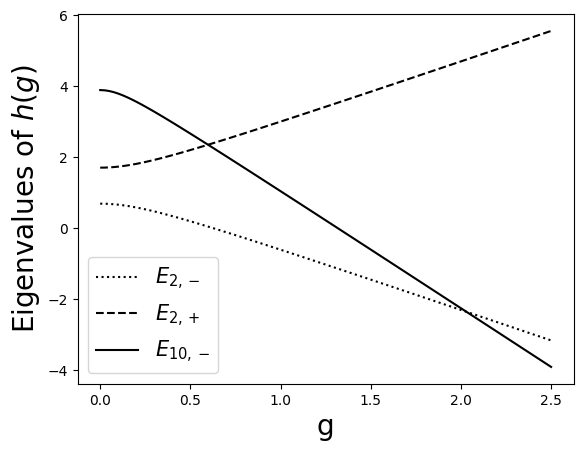}
         \caption{Case $(n,\nu)=(10,-)$}
         \label{fig:three sin x}
     \end{subfigure}
     
    \caption{
   Plot of some eigenvalues  $E_{(n,\nu)}(\cdot)$ when $(\omega,\Omega)=(2/5,\sqrt{2}).$ This illustrative example shows that, for each eigenvalue
   $E_{n,\nu}(\cdot)$, it is possible to find another eigenvalue $E_{m,\mu}(\cdot)$ such that $E_{n,\nu}(g)=E_{m,\mu}(g)$ at some positive $g$. For $(n,\nu)=(10,+)$, we can choose, for example, $(m,\mu)=(15,-)$ or $(m,\mu)=(20,-)$. Similarly, for $(n,\nu)=(10,-)$, we can choose $(m,\mu)=(2,-)$ or $(m,\mu)=(2,+)$.}
   \label{Fig:example_3}
\end{figure}
If $\nu=+$, there exists $m\in\mathbb{N}$ such that $E_{m,-}(0)>E_{n,\nu}(0)$. Since $E_{m,-}(\cdot)$ is decreasing, $E_{n,\nu}(\cdot)$ is increasing and they have 
 uniformly nonzero slope for $g\gg 1$,
we have that $E_{m,-}(g)=E_{n,\nu}(g)$ for some $g>0$.  If $\nu=-$ and $n\geq 1 $,  then for all $(m,\mu)\in(\mathbb{N}\times\{+,-\})\cup\{(-1,\delta)\}$ such that $E_{(m,\mu)}(0)<E_{n,\nu}(0)$, there exists $g\in\mathbb{R}$ such that $E_{(n,\nu)}(g)=E_{(m,\mu)}(g)$. On the other hand, if $\nu=-$ and $n\in\{-1,0\}$, then, for each $(m,\mu)$ such that $m\geq 1$ and $\nu=-$, there exists $g\in\mathbb{R}$ such that $E_{(n,\nu)}(g)=E_{(m,\mu)}(g)$. { Therefore, for each $(n,\nu)\in(\mathbb{N}\times\{+,-\})\cup\{(-1,\delta)\}$, the eigenvalue $E_{n,\nu}(\cdot)$ either intersects all eigenvalues $E_{m,\mu}(\cdot)$ such that $E_{m,\mu}(0)\leq E_{n,\nu}(0)$, or intersects a decreasing eigenvalue $E_{m,\mu}(\cdot)$ that itself intersects all eigenvalues $E_{m',\mu'}(\cdot)$ such that $E_{m',\mu'}(0)\leq E_{m,\mu}(0)$. We conclude that the spectrum of $h(\cdot)=H(\cdot,0)$ is always connected by eigenvalue intersections.} 
Illustrative examples are given in Figure~\ref{Fig:example_3}. 
Take $(n,\nu),(m,\mu)\in(\mathbb{N}\times\{+,-\})\cup\{(-1,\delta)\}$. If there exists $g\in\mathbb{R}$ such that $E_{(n,\nu)}(g)=E_{(m,\mu)}(g)$, 
then 
$x=g^{2}/\omega^{2}$ 
must solve
the  polynomial equation
    \begin{equation*}
        x^{2}-2(m+n+2)x+(m-n)^{2}-\left(\frac{\Delta}{\omega}\right)^{2}=0.
    \end{equation*}
    
We claim that if $(\Delta/\omega)^{2}$ is irrational then for every $g\in\mathbb{R}$ there  exists at most one pair of distinct eigenvalues $E_{(n,\nu)}(\cdot)$ and $E_{(m,\mu)}(\cdot)$ 
that intersect at $g$. 
Indeed, assume that there exist two pairs of distinct eigenvalues $\{E_{(n,\nu)}(\cdot),E_{(m,\mu)}(\cdot)\}$ and $\{E_{(n',\nu')}(\cdot),E_{(m',\mu')}(\cdot)\}$ intersecting at $g=\bar{g}$, with {$m\geq n$ and $m'\geq n'$}.
Hence,   $\bar{x}=(\bar{g}/\omega)^{2}$ satisfies
\begin{align}
         \bar{x}^{2}-2(m+n+2)\bar{x}+(m-n)^{2}-\left(\frac{\Delta}{\omega}\right)^{2}&=0,\label{eq:int_1}\\
         \bar{x}^{2}-2(m'+n'+2)\bar{x}+(m'-n')^{2}-\left(\frac{\Delta}{\omega}\right)^{2}&=0\label{eq:int_2}.
\end{align}
It is evident from  any of these two equations that $\bar{x}\notin\mathbb{Q}$. By subtracting $\eqref{eq:int_2}$ from $\eqref{eq:int_1}$, we  obtain that
\begin{equation*}
    2(m+n-m'-n')\bar{x}=(m-n)^{2}-(m'-n')^{2}.
\end{equation*}
Since $\bar{x}\notin\mathbb{Q}$, we deduce that $m=m'$ and $n=n'$. It can be simply checked that, in addition, $\nu=\nu'$ and $\mu=\mu'$. Thus the two pairs of eigenvalues are equal. 

We can deduce that, if $(\Delta/\omega)^2$ is irrational, the spectrum of $h(\cdot)=H(\cdot,0)$ is connected by non-overlapping eigenvalue intersections between eigenvalues. It can also be proven that the intersecting eigenvalues on $u_{2}=0$ are not tangent to each other. Moreover, by Lemma \ref{lemma:non-degenerate} ($\Omega/\omega$ is irrational since $(\Delta/\omega)^2$ is irrational), all the eigenvalues are simple when $u_{2}\neq 0$. Hence all eigenvalue intersections on $u_{2}=0$ are weakly conical with conical direction $\eta=(1,0)$.
We can conclude that for almost every $\Omega, \omega>0$ the spectrum 
of $H(\cdot,\cdot)$ 
is weakly conically connected and the eigenvalue intersections are not overlapping.
\end{proof}

\begin{proof}[Proof of Theorem \ref{theorem:Controllability_JC}]
 By Proposition \ref{prop:weakly-conical},  for almost every $\Omega,\omega>0$, the spectrum of $H(\cdot,\cdot)$ is weakly conically connected and has rationally unrelated germs at each intersection point (since the eigenvalue { intersections are} non-overlapping). 
 Since, moreover, the driven Jaynes--Cummings Hamitonian in system \eqref{eq:system_3} coincides with $H(g,u)$, then the conclusion follows from  Theorem~\ref{theorem:main}.
\end{proof}
Our result shows that for almost every (i.e., in a subset with zero-measure complement in $\mathbb{R}$) atomic transition frequency $\omega$, harmonic oscillator frequency $\Omega$, and (vacuum) Rabi frequency $g$, the corresponding Jaynes-Cummings model \eqref{eq:system_3} driven by the single time-dependent input $u(t)$, which is the amplitude of the resonant drive,  is approximately controllable { on} the unitary group. In other words, this means that any unitary operation can be realized with arbitrary precision { with respect to the strong
operator topology} at the expense of more complex control.

\section{Conclusion}
In this paper, we established a spectral condition for the controllability of quantum systems subject to a static field and a time-dependent control. The key assumption is that the spectrum of a two-input quantum system is connected by weakly conical intersections. 
We applied our result to two physical examples: a model of enantio-selective excitation and the driven Jaynes--Cummings model. 
These examples
highlight
how our result can be applied to study the controllability of more general scenarios: either by transforming the control of a multi-input system into the control via a static field and a time-dependent control, or by formally interpreting physical parameters in the system as static fields. 

\section*{Acknowledgments} 
\noindent This work has been partly supported by the ANR-DFG project “CoRoMo”
ANR-22-CE92-0077-01 and DFG project number 505622963 (KO 2301/15-1.) 
Financial support from the CNRS through the MITI interdisciplinary programs and 
from the Deutsche Forschungsgemeinschaft through CRC 1319 ELCH is gratefully acknowledged. This research has been funded in whole or in part by the French National Research Agency (ANR) as part of the QuBiCCS project "ANR-24-CE40-3008-01." 
\bibliography{reference}
\end{document}